\documentclass[a4paper,dvips,12pt,oneside]{article}
\usepackage[makeroom]{cancel}
\usepackage{graphicx}
\usepackage{epstopdf}
\usepackage{subfig}
\usepackage{color}
\usepackage{float}
\usepackage{leftidx}
\usepackage{bigints}
\usepackage{amssymb,amsmath,dsfont,url,epsfig}
\usepackage[left=1.5in, right=0.91in, top=1.1in, bottom=1.2in, includefoot, headheight=13.6pt]{geometry}
\usepackage[T1]{fontenc}
\usepackage{titlesec, blindtext, color}
\definecolor{gray75}{gray}{0.75}

\newcommand{\sln}{\linespread{1}}
\newcommand*{\email}[1]{\href{mailto:#1}{\nolinkurl{#1}} } 
\titleformat{\chapter}[block]{\LARGE\bfseries\sln}{Chapter \thechapter}{11pt}{\newline\huge\bfseries}
\newtheorem{theorem}{Theorem}[section]
\newtheorem{definition}{Definition}[section]

\newenvironment{proof}{\paragraph{Proof:}}{\hfill$\square$}
\newtheorem{lemma}{Lemma}[section]

\begin{document}
\title{ON C3-LIKE FINSLER METRICS UNDER RICCI FLOW}
\author{Ranadip Gangopadhyay and Bankteshwar Tiwari\\
DST-CIMS, Institute of Science, Banaras Hindu University,\\ Varanasi-221005, India}

\maketitle

\thanks{This paper is accepted for publication in Bulletin of the Transilvania University of Brasov  Series III -   Mathematics and Computer Science}
\begin{abstract}
 In this paper we have studied the class of Finsler metrics, called C3-like metrics which satisfy the un-normal and normal Ricci flow equation and proved that such metrics are Einstein.
\end{abstract}

\section{Introduction}
The Ricci flow theory, which was introduced by  Richard Hamilton, became a very powerful method in understanding the geometry and topology of Riemannian manifolds. Finally in the first decade of twenty first century Perelman able to classify the three dimensional smooth manifolds using Ricci flow equation \cite{GP1},\cite{GP2},\cite{GP3}. The Ricci flow equation is the evolution equation
\begin{center}
$\frac{d}{dt} g_{ij} = -2Ric_{ij} $,
\end{center}
for a Riemannian metric $g_{ij}$, where $Ric_{ij}$ is the Ricci curvature tensor. Hamilton showed that there is a unique solution to this equation for an arbitrary smooth metric on a closed manifold over a sufficiently short time \cite{H1},\cite{H2}.
\\The concept of Einstein metric is also very important topic in differential geometry as well as in physics. A Riemannian manifold $(M,g)$ is called Einstein if it's Ricci tensor $Ric$ can be given by $Ric = kg$, where $k$ is some constant. It is named after Albert Einstein because this condition is equivalent to saying that the metric is a solution of the vacuum Einstein field equations. A Finsler manifold $(M,F)$ is called Einstein if there exist a scalar function $K(x)$ on $M$ such that $Ric=(n-1)KF^2$.\\
 Finsler metrics satisfying Ricci flow equation has been an important topic of research. For instance Sadegzadeh and Razavi have studied C-reducible metrics satisfying Ricci flow equation \cite{NSAR}, where as Tayebi, Payghan and Najafi have studied semi C-reducible Finsler metrics satisfying Ricci flow equation \cite{TPN} and proved that such metrics are Einstein when they satisfies the Ricci flow equation.\\
 In the present paper we have extended those results to a more general class of Finsler metric, namely C3-like Finsler metric, introduced by Prasad and Singh \cite{PS}. More precisely we have prove the following theorems:
\begin{theorem}
Let $(M,F)$ be a Finsler manifold of dimension $n \ge 3$. Suppose that $F$ be a C3-like Finsler metric, then every deformation $F_t$ of the metric $F$ satisfying un-normal Ricci flow equation is an Einstein metric.
\end{theorem}
\begin{theorem}
Let $(M,F)$ be a Finsler manifold of dimension $n \ge 3$. Suppose that $F$ be a C3-like Finsler metric, then every deformation $F_t$ of the metric $F$ satisfying normal Ricci flow equation is an Einstein metric.
\end{theorem}

\section{Preliminaries}

Let $M$ be an $n$-dimensional smooth manifold and $T_{x}M$ denotes the tangent space of $M$
 at $x$. The tangent bundle of $ M $ is the union of tangent spaces $ TM:= \bigcup _{x \in M}T_xM $. We denote the elements of $TM$ by $(x,y)$ where $y\in T_{x}M $. Let $TM_0:=TM \setminus\left\lbrace 0\right\rbrace $.
 \begin{definition}\cite{SSZ}
 A Finsler metric on $M$ is a function $F:TM\rightarrow[0,\infty)$ with the following properties:
 \\(i) $F$ is $C^\infty$ on $TM_{0}$,
 \\(ii) $F$ is a positively 1-homogeneous on the fibers of tangent bundle $TM$.
 \\(iii) The Hessian of $\frac{F^2}{2}$ with element $g_{ij}=\frac{1}{2}\frac{\partial ^2F^2}{\partial y^i \partial y^j}$ is positive definite on $TM_0$.
 \\The pair $(M,F)$ is called a Finsler space, $F$ is called the fundamental function and $g_{ij}$ is called the fundamental tensor.
 \end{definition}
 Let $\left( M,F \right) $ be a Finsler space. For a vector $y \in T_xM \setminus \left\lbrace  0\right\rbrace $, let 
 \begin{center}
 $C_y(u,v,w) := \frac{1}{4} \frac{\partial ^3}{\partial s \partial t \partial r}[F^2(y+su+tv+rw)]_{s=t=r=0}$
 \end{center}
where $u,v,w \in T_xM$. Each $C_y$ is a symmetric trilinear form on $V$. We call the family $C:= \left\lbrace C_y : y \in T_xM \setminus \left\lbrace  0\right\rbrace\right\rbrace  $ the Cartan torsion. Let $\left\lbrace b_i \right\rbrace $ be a basis for $T_xM$ and define $g_{ij} := g_y(b_i,b_j)$, $C_{ijk} := C_y(b_i,b_j,b_k)$. 
 \\ Then $C_{ijk} = \frac{1}{4}[F^2]_{y^iy^jy^k} = \frac{1}{2}\frac{\partial}{\partial y^k}(g_{ij})$.
 \\Define the mean value of the Cartan torsion by
 \begin{center}
 $I_y(u) := \Sigma _{i=1}^{n}g^{ij}(y)C_y(u,b_i,b_j)$ , $u \in T_xM$.
 \end{center}
We call the family $I := \left\lbrace I_y | y \in T_xM \setminus \left\lbrace  0\right\rbrace \right\rbrace $ the mean Cartan torsion at $x \in M$  
  \begin{center}
 i.e. $I_i = g^{jk}C_{ijk}$.
  \end{center}
  \begin{definition}\cite{PS}
 A Finsler metric $F$ is called C3-like if its Cartan torsion is given by 
  \begin{equation}\label{eq0}
  C_{ijk}=\left\lbrace a_ih_{jk}+ a_jh_{ki}+a_kh_{ij} \right\rbrace +\left\lbrace b_iI_jI_k+b_jI_iI_k+b_k I_iI_j \right\rbrace,
  \end{equation}
  where $a_i=a_i(x,y)$ and $b_i=b_i(x,y)$ are covectors on $TM$ and of degree $-1$ and $1$, respectively and $h_{ij}$ is angular metric tensor given by $h_{ij}=F\frac{\partial^2 F}{\partial y^i \partial y^j}$.
    \end{definition}
     In particular :
  \\ (i)   if $a_i = 0$, then we have $C_{ijk} = {b_iI_jI_k +b_jI_i I_k + b_kI_iI_j }$, contracting it with $g^{ij}$ implies that $b_i = \frac{1}{3I^2}I^i$, then $F$ is a C2-like Finsler metric,
  \\ (ii)  if $b_i = 0$, then
  we have $C_{ijk} = a_ih_{jk} + a_jh_{ki} + a_kh_{ij}$, contracting it with $g^{ij}$ implies that $a_i = \frac{1}{n+1}I_i$, then $F$ is a C-reducible Finsler metric.
  \\ (iii) if $a_i =\frac{p}{n + 1}I_i$ and $b_i =\frac{q}{3I^2}I_i$, where $p = p(x, y)$ and $q = q(x, y)$ are scalar functions
  on $TM$, then $F$ is a semi C-reducible Finsler metric.
  \\ Therefore, a C3-like Finsler metric may be consider as a generalization of C-reducible, semi C-reducible and C2-like Finsler metrics.
\section{Un-normal Ricci flow equation on C3-like Finsler metrics}

The geometric evolution equation 
\begin{equation}\label{eq1}
\frac{d}{dt} g_{ij} = -2Ric_{ij},
\end{equation}
is known as the un-normalized Ricci flow in Riemannian geometry. In principle, the same equation can be used in Finsler setting, because both $g_{ij}$ and $Ric_{ij}$ have been generalised to the broader framework, albeit gaining a $y$ dependence in the process. However, there are two reasons why we shall refrain from doing so.
\\ (i) Not every symmetric covariant 2-tensor $g_{ij}(x,y)$ arises from a Finsler metric $F(x,y)$.
\\(ii) There is more than one geometrical context in which $g_{ij}$ makes sense.
\par A deformation of Finsler metrics means a $1$-parameter family of metrics $g_{ij}(x,y,t)$, such that $t \in [-\epsilon, \epsilon]$ and $\epsilon > 0$ is sufficiently small. For such a metric $\omega = u_idx^i$, the volume element as well as the connections attached to it depend on $t$. The same equation can be used in the Finsler setting. But instead of the above tensor evolution equation, we will use a different form of the above equation. By contracting $\frac{d}{dt} g_{ij} = -2Ric_{ij} $ with $y^i$ and $y^j$ respectively and using Eulers theorem, we get 
\begin{center}
$\frac{\partial F^2}{\partial t} = -2F^2R $,
\end{center}
where $R= \frac{Ric}{F^2}$. That is,
\begin{equation*}
d \log F =-R , F(t=0)=F_0.
\end{equation*}
This scalar equation directly addresses the evolution of the Finsler metric $F$,
and makes geometrical sense on both the manifold of nonzero tangent vectors
$TM_0$ and the manifold of rays. It is therefore suitable as an un-normalized Ricci
flow for Finsler geometry.
In this section we will study C3-like Finsler metrics satisfying un-normal Ricci flow equation. 
\\ Let us assume that $F_t$ be a deformation of C3-like Finsler metrics which satisfies the Un-Normalised Ricci flow equation given by 
\begin{equation*}
g_{ij}^{\prime}= -2Ric_{ij} \quad \textnormal{or} \quad \frac{F^{\prime}}{F}=-R, \quad \textnormal{where} \quad g_{ij}'=\frac{\partial g_{ij}}{\partial t}
\end{equation*}
 Before proving Theorem 1.1 we need the following lemmas:
 \begin{lemma}
If $F_t$ be a deformation of a C3-like Finsler metrics $F$ on manifold $M$ of
dimension $n \ge 3$, then the variation of Cartan torsion is given by 
\begin{equation}\label{eq2}
C_{ijk}^{\prime}I^iI^jI^k =3(a_i+3b_i)I^i \|I\|^2+\frac{1}{2}F^2R_{,i,j,k}I^iI^jI^k+3\ \|I\|^2I_mR_{,m},
\end{equation}
where $R_{,m}= \frac{\partial R}{\partial y^m}.$
 \end{lemma}
\begin{proof}
 Let us assume that $F_t$ be a deformation of a Finsler metric F on a 2-dimensional manifold $M$ satisfies Ricci flow equation given by (\ref{eq1}).\\
By definition of Ricci tensor, we have
\begin{eqnarray}\label{eq3}
Ric_{ij}&=& \frac{1}{2}[RF^2]_{y^iy^j} \nonumber\\
&=& Rg_{ij}+\frac{1}{2}F^2R_{,i,j}+R_{,i}y_{j}+R_{,j}y_i
\end{eqnarray}
where $R_{,i}=\frac{\partial R}{\partial y^i} $ and $R_{,i,j}=\frac{\partial ^2 R}{\partial y^i \partial y^j}$.\\ Taking vertical derivative of (\ref{eq3}) and from the fact $y_{i,j}=g_{ij}$ and $FF_k = y_k$ we have \\
\begin{eqnarray}\label{eq4}
Ric_{ij,k}= 2RC_{ijk}+\frac{1}{2}F^2R_{,i,j,k}+(g_{jk}R_{,i}+g_{ij}R_{,k}+g_{ki}R_{,j}) \nonumber \\ +\left\lbrace R_{,j,k}y_i+R_{,i,j}y_k+R_{,k,i}y_j\right\rbrace .
\end{eqnarray}
Multiplying (\ref{eq4}) with $I^iI^jI^k$ and using $y_iI^i=y^iI_i=0$ we obtain
\begin{equation}\label{eq5}
Ric_{ij,k}I^iI^jI^k= 2RC_{ijk}I^iI^jI^k+\frac{1}{2}F^2R_{,i,j,k}I^iI^jI^k+3\ \|I\|^2I_mR_{,m}
\end{equation}
Now multiplying (\ref{eq0}) by $I^iI^jI^k$ we have
\begin{equation}\label{eq25}
C_{ijk}I^iI^jI^k =3(a_i+3b_i)I^i \|I\|^2.
\end{equation}
Since $F_t$ satisfies the Ricci flow equation so using  (\ref{eq25}) we have
\begin{equation}\label{eq6}
C_{ijk}' I^iI^jI^k=3(a_i+3b_i)I^i \|I\|^2+\frac{1}{2}F^2R_{,i,j,k}I^iI^jI^k+3\ \|I\|^2I_mR_{,m}.
\end{equation}
\end{proof}
\begin{lemma}
If $F_t$ be a deformation of C3-like Finsler metrics $F$ on a manifold $M$ of
dimension $n \ge 3$, then $C_{ijk}'I^iI^jI^k$ is divisible by $\|I\|^2$. 
\end{lemma}
\begin{proof}
Since $g^{ij}g_{jk} = \delta_k ^i$, then differentiating with respect to $t$ and using $g'_{ij}= -2Ric_{ij}$ we have $g'^{il}=2Ric^{il}$.
\\At first we will calculate the value of $I_i'$
\begin{eqnarray}\label{eq7}
I_i' \nonumber
&=& (g^{jk}C_{ijk})'
\\&=& (g^{jk})'C_{ijk} +g^{jk}C_{ijk}' \nonumber
\\ &=& 2Ric^{jk}C_{ijk}-g^{jk}Ric_{jk,i} \nonumber
\\ &=& Ric ^{jk}g_{jk,i}-(g^{jk}Ric_{jk})_{,i}+ g^{jk}_{,i}Ric_{jk} \nonumber
\\&=& -(g^{jk}Ric_{jk}),i \nonumber
\\&=& -\rho_i
\end{eqnarray}
where $\rho := g^{jk}Ric_{jk}$ and $\rho_i = \frac{\partial \rho}{\partial y^i}$.
\\Since  $y_i:= FF_{y^i}$, differentiating  with respect to $t$ we have
\begin{center}
$y_i'= -2Ric_{im}y^m$.
\end{center}
Now we compute $h_{ij}'$.
\begin{eqnarray}\label{eq8}
h_{ij}'\nonumber 
&=& (g_{ij}-F^{-2}y_iy_j)'\nonumber
\\ &=& g_{ij}'- \left[ -2F^{-3}F'y_iy_j+F^{-2}y_i'y_j+F^{-2}y_iy_j'\right] \nonumber
\\ &=& -2Ric_{ij}+2F^{-2}\frac{F'}{F}y_iy_j-F^{-2}(y_i'y_j+y_iy_j') \nonumber
\\&=& -2Ric_{ij}+2F^{-2}\frac{F'}{F}y_iy_j+2F^{-2}(Ric_{im}y_j+Ric_{jm}y_i)y^m \nonumber
\\ &=&-2Ric_{ij}-2Rl_il_j+2(Ric_{im}l_j+Ric_{jm}l_i)l^m\nonumber
\\&=&-2Ric_{ij}+2R(h_{ij}-g_{ij})+2(Ric_{im}l_j+Ric_{jm}l_i)l^m
\end{eqnarray}
where $l_i = \frac{y_i}{F}$.
\\ Multiplying (\ref{eq0}) by $g^{ij}$ we have
\begin{equation}\label{eq10}
a_i=\frac{1}{n+1}\left\lbrace \left( 1-2I^mb_m\right) I_i -\|I^2\|b_i\right\rbrace.
\end{equation}
\\After differentiating (\ref{eq0}) with respect to $t$ we have
\\\begin{eqnarray}\label{eq21}
C_{ijk}'=(a_ih_{jk}'+a_jh_{ki}'+a_kh_{ij}') + (a_i'h_{jk}+a_j'h_{ki}+a_k'h_{ij}) \hspace{2.0cm} \nonumber \\ -(b_i(\rho_jI_k+I_j\rho_k)  +b_j(\rho_iI_k+I_i\rho_k)+b_k(\rho_iI_j+I_i\rho_j)). 
\end{eqnarray}
\\ Now contracting (\ref{eq21}) with $I^iI^jI^k$ we have 
\\ \begin{equation}\label{eq22}
C_{ijk}'I^iI^jI^k =\left( a_ih'_{jk}+a_jh'_{ki}+a_kh'_{ij}\right) I^iI^jI^k +3a'_iI^i\|I\|^2-6b_i\rho_j\|I\|^2I^iI^j.
\end{equation}
\\ Using (\ref{eq8}) and (\ref{eq10}) in (\ref{eq22}) we obtain
\begin{equation}\label{eq12}
C_{ijk}'I^iI^jI^k = \frac{-6}{n+1} Ric_{ij}I^iI^j\left(1-3b_mI^m\right)\|I^2\| +3a'_iI^i\|I^2\| - 6b_i\rho_j\|I^2\|I^iI^j.
\end{equation}
\\ Therfore the result follows.
\end{proof}

\textbf{Proof of Theorem 1.1} 

In the view of (\ref{eq6}) and (\ref{eq12}) it follows that $R_{,i,j,k}I^iI^jI^k$ is divisible by $\|I^2\|$. Then we can write
\begin{eqnarray}\label{eq13}
R_{,i,j,k} = A_{ij}I_k + B_ig_{jk}.
\end{eqnarray}
By contracting $R_{,i,j,k}$ with $y^j$ or $y^k$ we have $R_{,i}=0$. So $F_t$ is Einstein.
\section{Normal Ricci flow equation on C3-like Finsler metrics}
If $M$ is compact, then so is $SM$, and we can normalize the above equation by requiring that the flow keeps the volume of $SM$ constant. Recalling the Hilbert form $\omega := F_{y^i}dx^i$, that volume is 
\begin{equation*}
Vol_{SM} := \int_{SM} \frac{(-1)^{\frac{(n-1)(n-2)}{2}}}{(n-1)!} \omega \wedge (d\omega)^{(n-1)} := \int_{SM}dV_{SM}.
\end{equation*}
During the evolution, $F$, $\omega$ and consequently the volume form $dV_{SM}$ and the
volume $Vol_{SM}$, all depend on $t$. On the other hand, the domain of integration
$SM$, being the quotient space of $TM_0$ under the equivalence relation $z \equiv y$,
$z = \lambda y$ for some $\lambda > 0$, is totally independent of any Finsler metric and hence does not depend on $t$. We have from insights in \cite{AKZD}
\begin{center}
$\frac{d}{dt}(dV_{SM}) = \left[ g_{ij}-g_{ij}'-n\frac{d}{dt} \log F\right] dV_{SM}$.
\end{center}
A normalized Ricci flow for Finsler metrics is proposed by Bao and it is given by 
\begin{equation}\label{eq23}
\frac{d}{dt} \log F = -R +\frac{1}{Vol(SM)}\int_{SM} RdV , \quad F(t=0)=F_0, 
\end{equation}
where the given manifold $M$ is compact. Now, we let $Vol(SM)= 1$. Then
all of Ricci-constant metrics are exactly the fixed points of the above flow. Let
\begin{equation} \label{eq14}
Ric_{ij} = \frac{1}{2}(F^2R)_{.y^i.y^j}.
\end{equation}
Differentiating (\ref{eq23}) with respect to $y^i$ and $y^j$ we have, 
\begin{equation}\label{eq15}
\frac{d}{dt}(g_{ij}) = -2Ric_{ij}+ \frac{2}{Vol(SM)}\int_{SM}RdVg_{ij},\quad g(t=0)=g_0.
\end{equation}
Starting with any familiar metric on M as the initial data $F_0$, we may deform
it using the proposed normalized Ricci flow, in the hope of arriving at a Ricci
constant metric.
\\\textbf{Proof of Theorem 1.2} 
\\Here we consider Finsler manifolds that satisfies the normal Ricci flow equation. Then ,
\begin{eqnarray}\label{eq16}
\frac{dg_{ij}}{dt} = -2Ric_{ij}+ 2\int_{SM}RdVg_{ij} ,\quad  d(\log F) = \frac{F'}{F}=-R+\int_{SM}RdV.
\end{eqnarray}
In this case,
\begin{equation}
y_i'= \left( -2Ric_{im}+2\int_{SM}RdVg_{im}\right)y^m.
\end{equation}
Therefore, 
\begin{eqnarray}\label{eq17}
I_i' \nonumber
&=& (g^{jk}C_{ijk})'
\\&=& (g^{jk})'C_{ijk} +g^{jk}C_{ijk}' \nonumber
\\ &=& 2\left[ Ric^{jk}- \int_{SM}RdVg^{jk}]C_{ijk}+g^{jk}[Ric_{jk,i}+2\int_{SM}RdVC_{ijk}\right]  \nonumber
\\ &=& 2Ric ^{jk}g_{jk,i}-(g^{jk}Ric_{jk})_{,i}+ g^{jk}_,Ric_{jk} \nonumber
\\&=& -(g^{jk}Ric_{jk}),i \nonumber
\\&=& -\rho_i.
\end{eqnarray}
\\And  
\begin{eqnarray} \label{eq18}
h_{ij}'\nonumber 
&=& (g_{ij}-F^{-2}y_iy_j)'\nonumber
 \\ &=& g_{ij}'- \left[ -2F^{-3}F'y_iy_j+F^{-2}y_i'y_j+F^{-2}y_iy_j'\right] \nonumber
 \\ &=& g_{ij}'+2F^{-2}\frac{F'}{F}y_iy_j -F^{-2}\left( y_i'y_j+y_iy_j'\right)   \nonumber
\\&=&-2Ric_{ij}+2\int_{SM}RdVg_{ij}+2\left( -R-\int_{SM}RdV\right) l_il_j \nonumber \\ &&+2(Ric_{im}l_j+Ric_{jm}l_i)l^m\nonumber
\\&=&-2Ric_{ij}+2\int_{SM}RdVg_{ij}+2(-R-\int_{SM}RdV)(g_{ij}-h_{ij})\nonumber \\ &&+2(Ric_{im}l_j+Ric_{jm}l_i)l^m\nonumber
\\&=&-2Ric_{ij}-2R(h_{ij}-g_{ij})+2\int_{SM}RdVh_{ij} \nonumber \\ &&+2(Ric_{im}l_j+Ric_{jm}l_i)l^m.
\end{eqnarray}
Therefore,
\begin{equation} \label{eq24}
h_{ij}=-2Ric_{ij}-2R(h_{ij}-g_{ij})+2\int_{SM}RdVh_{ij}+2(Ric_{im}l_j+Ric_{jm}l_i)l^m.
\end{equation}
\\ Now from the un-normalised case,
\\
\begin{eqnarray}\label{eq19}
C_{ijk}'I^iI^jI^k=(a_ih_{jk}'+a_jh_{ki}'+a_kh_{ij}')I^iI^jI^k + (a_i'h_{jk}+a_j'h_{ki}+a_k'h_{ij})I^iI^jI^k \nonumber \\-\left\lbrace b_i(\rho_jI_k+I_j\rho_k)+b_j(\rho_iI_k+I_i\rho_k) +b_k(\rho_iI_j+I_i\rho_j)\right\rbrace I^iI^jI^k.
\end{eqnarray}
\\Now using (\ref{eq18}) and (\ref{eq10}) we get
\\\begin{eqnarray}\label{eq20}
C'_{ijk}I^iI^jI^k = \frac{-6}{n+1}Ric_{ij}I^iI^j(1-3b_mI^m)\|I^2\|+3a_i'I_i\|I^2\| \nonumber \\ -6b_i\rho_jI^iI^j\|I^2\|+6a_k\int_{SM}RdVI^k \|I^2\|.
\end{eqnarray}
So by the same arguement of un-normalised case the result follows.

\end{document}